\pdfoutput=1
\documentclass{article}

\usepackage{microtype}
\usepackage{graphicx}
\usepackage{subfigure}
\usepackage{booktabs} 
\usepackage{amsmath,amsthm,amssymb,amsfonts,bbm,bm,math}

\usepackage{mathtools,amssymb,lipsum}

\usepackage{cuted}
\usepackage{hyperref}
\makeatletter
\newcommand*{\rom}[1]{\expandafter\@slowromancap\romannumeral #1@}
\makeatother

\newtheorem{lemm}{Lemma S\hspace{-2.5pt}}

\usepackage[accepted]{icml2018}

\icmltitlerunning{Dissipativity Theory for Accelerating Stochastic Variance Reduction}

\begin{document}

\twocolumn[
\icmltitle{Dissipativity Theory for Accelerating Stochastic Variance Reduction: \\
  A Unified Analysis of SVRG and Katyusha Using Semidefinite Programs}



\icmlsetsymbol{equal}{*}
\begin{icmlauthorlist}
\icmlauthor{Bin Hu}{ed}
\icmlauthor{Stephen Wright}{ed}
\icmlauthor{Laurent Lessard}{ed}
\end{icmlauthorlist}
\icmlaffiliation{ed}{University of Wisconsin--Madison, United States}
\icmlcorrespondingauthor{Bin Hu}{bhu38@wisc.edu}

\icmlkeywords{Machine Learning, ICML}

\vskip 0.3in
]

\printAffiliationsAndNotice{}

\begin{abstract}
Techniques for reducing the variance of gradient estimates used in
stochastic programming algorithms for convex finite-sum problems have
received a great deal of attention in recent years. By leveraging
dissipativity theory from control, we provide a new perspective on two
important variance-reduction algorithms: SVRG and its direct
accelerated variant Katyusha.  Our perspective provides a physically
intuitive understanding of the behavior of SVRG-like methods via
a principle of energy conservation.  The tools discussed
here allow us to automate the convergence analysis of SVRG-like
methods by capturing their essential properties in small semidefinite
programs amenable to standard analysis and computational techniques.  Our approach recovers existing convergence results
for SVRG and Katyusha and generalizes the theory to alternative parameter choices.  We also discuss how our approach complements the linear
coupling technique. Our combination of perspectives leads to
a better understanding of accelerated variance-reduced stochastic
methods for finite-sum problems.
\end{abstract}

\section{Introduction}

Empirical risk minimization (ERM) is a key paradigm in machine
learning \cite{bubeck2015, bottou2016optimization}. Many learning
problems, including ridge regression, logistic regression, and support
vector machines, can be naturally formulated as the following
finite-sum ERM
\begin{align}
\label{eq:minP}
\min_{x\in \R^p} \,\, g(x)\defeq\frac{1}{n} \sum_{i=1}^n f_i(x),
\end{align}
where $g$ is strongly convex.
A standard approach for solving~\eqref{eq:minP} is the stochastic
gradient (SG) method \cite{robbins1951, Bottou2004}.  Recently, a
large family of variance-reduction methods have been developed to
improve the convergence guarantees of SG.  Such methods are typically
classified into the following two categories:
\begin{itemize}
\item[1.]  SVRG-like methods are epoch-based, requiring evaluation of
  a complete gradient $\nabla g(\tilde{x})$ at the beginning of each
  epoch. The epoch length is typically set to be $2n$ but can be
  adaptive. Methods of this type include SVRG
  \cite{johnson2013} and its direct accelerated variant Katyusha
  \cite{katyusha2016}. 
\item[2.] SAGA-like methods do not involve epoch length tuning, and
  include SAG \cite{Roux2012, Schmidt2013}, SAGA \cite{defazio2014},
  Finito \cite{defazio2014finito}, SDCA \cite{shalev2013, Shai2016},
  APCG \cite{lin2014APCG}, SPDC \cite{zhang2017SPDC}, and point-SAGA
  \cite{defazio2016simple}. Due to storage issues, it may be difficult
  to apply SAGA-like methods to general learning problems other than
  linear prediction/classification, since their required storage for general learning tasks scales with the training set size.
\end{itemize}

This paper is motivated by the following two concerns. First, there
has been recent interest in developing a unified, coherent set of
tools for analyzing stochastic finite-sum methods. Traditionally,
convergence proofs for variance-reduction methods have been developed
in a case-by-case manner. More coherent techniques may facilitate the
design of new finite-sum methods in more complicated setups. Recently,
control theory has been used to derive linear matrix inequality (LMI)
conditions that can be used to automate the analysis of a large
family of first-order optimization methods \citep{Lessard2014, hu2017SGD,
  Bin2017COLT, BinHu2017, fazlyab2017analysis}.  Specifically, \citet{Bin2017COLT} has
tailored jump system theory to provide a unified analysis for SAGA,
Finito, and SDCA.  The analysis in \citet{Bin2017COLT} can potentially
be extended to cover other SAGA-like methods, such as SAG, APCG, SPDC,
and point-SAGA. However, as pointed out in \citet{Bin2017COLT}, jump
system theory may not be the most suitable tool for epoch-based
methods. This paper aims in part to bridge the gap between
control-oriented analysis and SVRG-like methods by extending the
deterministic dissipativity theory in \citet{BinHu2017} to a
stochastic setup. The approach of this paper allows us to formulate
semidefinite programs for a unified analysis of SVRG-like methods. Together,
dissipativity theory  and the jump system theory
described in \citet{Bin2017COLT}  provide a complete picture
of how control theory can be used to unify the analysis of stochastic
finite-sum methods.

Second, there is still a need for better understanding of the role of
momentum in the algorithms for the finite-sum problem \eqref{eq:minP}.
Nesterov's accelerated method~\citep{YEN03a} has received a great deal
of attention for its ingenuity and its appealing theoretical and
practical behavior. However, the original convergence rate proof of
Nesterov's accelerated method relies on a technique of estimate
sequences, and is not easy to interpret. Recently, new interpretations
of Nesterov's accelerated method have been proposed from many
different perspectives, for example, linear coupling
\cite{allen2014linear}, geometric descent \citep{bubeck2015geometric},
control theory \cite{Lessard2014, BinHu2017}, continuous-time ODEs
\cite{Su2016, wibisono2016, wilson2016}, and quadratic averaging
\cite{drusvyatskiy2016}.  Among these new developments, linear
coupling is the only one that has been extended to accelerate
variance-reduction methods for the finite-sum problem \eqref{eq:minP};
Katyusha momentum \cite{katyusha2016} is based on this idea. 
 Our
current paper extends the control-oriented approach in
\citet{BinHu2017} to cover accelerated variance-reduction methods.
The linear coupling framework of \citep{allen2014linear,katyusha2016} provides useful and intuitive design guidelines for accelerating optimization methods.
Our control approach complements linear coupling
 by providing a physical
interpretation of accelerated variance-reduction methods, as well as automated
convergence analysis via formulation and solution of small semidefinite programs.
Both linear coupling and our control approach provide useful
perspectives, and each has certain advantages from the viewpoint of
analysis.

Our contributions can be summarized as follows. We present a unified
analysis of SVRG and Katyusha by using the physically intuitive notion
of dissipativity. We prove convergence results for SVRG by solving a
$3 \times 3$ semidefinite program, and show that the existing
convergence result for Katyusha can be recovered and generalized by
solving a $6\times 6$ semidefinite program.  Numerical solutions of
our proposed LMIs can be used to narrow the choices for various
algorithm parameters (such as learning rate, momentum, and epoch
length) at early stages of proof construction.  We also present an
energy-conservation interpretation for variance reduction and
acceleration. Compared with \citet{BinHu2017}, the novelty of the present
paper is the development of several new stochastic supply rate
conditions that depend on the stochastic variance reduction mechanism.

\section{Preliminaries}

\subsection{Notation}
Let $\R$ and $\R_+$ denote the real and nonnegative real numbers, respectively.
We denote the $p\times p$ identity matrix as $I_p$.  The Kronecker
product of two matrices is denoted as $A \otimes B$.  Note hat
$(A\otimes B)^\tp=A^\tp \otimes B^\tp$ and $(A\otimes B)(C\otimes
D)=(AC)\otimes (BD)$ when the matrices have compatible dimensions.  A
differentiable function $f:\R^p\to \R$ is $\sigma$-strongly convex if
$f(x)\ge f(y)+\nabla f(y)^\tp (x-y)+\frac{\sigma}{2} \|x-y\|^2$ for
all $x,y \in \R^p$ and is $L$-smooth if $\|\nabla f(x)-\nabla
f(y)\|\le L \|x-y\|$ for all $x,y\in \R^p$. Note that $f$ is convex if
$f$ is $0$-strongly convex.  We use $x_\star$ to denote a point
satisfying $\nabla f(x_\star)=0$. When $f$ is $L$-smooth and
$\sigma$-strongly convex for some $\sigma>0$, $x_\star$ is unique.

\subsection{Dissipativity Theory for Stochastic Linear Systems}
\label{sec:DisTh}

For completeness, we first review dissipativity
theory for linear time-invariant (LTI) systems with stochastic inputs.
Our development parallels that of \citet[Section 2.2]{BinHu2017},
which reviews dissipativity theory for LTI systems with deterministic
inputs.

Consider an LTI system governed by the state-space model
\begin{align}
\label{eq:sys1}
\xi_{k+1}=A\xi_k+B w_k,
\end{align}
where $\xi_k\in \R^{n_\xi}$ is the state, $w_k\in \R^{n_w}$ is the
input, and $(A,B)$ are constant matrices with compatible
dimensions, i.e. $A\in \R^{{n_\xi}\times n_\xi}$ and $B\in
\R^{{n_\xi}\times n_w}$.  The input sequence $\{w_k\}$ is assumed to
be a stochastic process. 
Intuitively, we can interpret $w_k$ as a stochastic force driving the
state of the LTI model~\eqref{eq:sys1}.  Dissipativity theory
describes how the input forces $w_j$, $j=0,1,2,\dotsc$ drive the
internal energy stored in the states $\xi_k$, $k=0,1,2,\dotsc$.  The
theory hinges on two functions: a {\em supply rate}
$S:\R^{n_\xi}\times \R^{n_w}\to \R$ and a {\em storage function}
$V:\R^{n_\xi}\to \R_+$.  Since $w_k$ is stochastic, we adopt the
following notion of almost sure dissipativity.

\begin{defn}
  The system \eqref{eq:sys1} is almost surely~(a.s.) dissipative with
  respect to the supply rate $S:\R^{n_\xi}\times \R^{n_w}\to \R$ if
  there exists a storage function $V:\R^{n_\xi}\to \R_+$ such that
\begin{align}
\label{eq:DI}
V(\xi_{k+1})- V(\xi_k)\le S(\xi_k, w_k) \,\,\,\mbox{a.s.}
\end{align}
for all $k$. The
inequality \eqref{eq:DI} is called an a.s. dissipation inequality.
\end{defn}

We now discuss physical interpretations for the supply rate $S$, the
storage function $V$, and the dissipation
inequality~\eqref{eq:DI}. The storage function $V$ quantifies the
amount of internal energy stored in the system state $\xi_k$.  The
supply rate function $S$ maps any state/input pair $(\xi, w)$ to a
scalar that characterizes the energy supplied from the input $w$ to
the state $\xi$. (Note that the supply rate can be negative, in which
case the force $w_k$ is extracting energy from the system.) The
a.s. dissipation inequality \eqref{eq:DI} states that there will
always (technically ``a.s.'') be some energy dissipating from the
system \eqref{eq:DI}, and hence the internal energy increase, which is
$V(\xi_{k+1})-V(\xi_k)$, is bounded above by the energy supplied to
the system. The dissipation inequality can be thought of as a
restatement of the energy conservation law.
A useful variant of \eqref{eq:DI} is the exponential dissipation inequality:
\begin{equation}\label{eq:DI1}
V(\xi_{k+1})- \rho^2 V(\xi_k)\le S(\xi_k, w_k) \,\,\,\mbox{a.s.},
\end{equation}
where $0\le \rho \le 1$ is given.  The exponential dissipation inequality \eqref{eq:DI1} just states that at least a fraction $(1-\rho^2)$ of the internal energy will dissipate at every step $k$.

\begin{rem}
\label{rem:remark1}
It is often the case that the driving force $w_k$ depends on the state
$\xi_k$ in some prescribed way, so we know some properties
of the supply rate in advance.  If the supply rate function $S$
satisfies certain bounds, then the dissipation inequality can be used
to obtain convergence guarantees for \eqref{eq:sys1}.  For example, if
we know there exists a positive constant $M$ such that $\mathbb{E}S(\xi_k,
w_k)\le M$ for all $k$, then taking expectation of \eqref{eq:DI1}
leads to the conclusion that $\mathbb{E}V(\xi_{k+1})\le \rho^2
\mathbb{E}V(\xi_k)+M$.  Based on this inequality, one can show that 
$\mathbb{E}V(\xi_k)\le \rho^{2k} V(\xi_0)+\frac{M}{1-\rho^2}$. This
suggests that the state $\xi_k$ linearly converges to a ball
centered at the origin, and the radius of the ball is related to
$\frac{M}{1-\rho^2}$.  Later we will demonstrate that the convergence
of SG can be proved from a dissipation inequality argument of this
type.
\end{rem}
A computational advantage of dissipativity theory is that if the
supply rate $S$ is quadratic, we can search over admissible quadratic
storage functions $V$ by solving a small semidefinite program.  The
following approach is standard in the controls literature.
See \citet{willems72a, willems72b,
    willems2007dissipative} for a more comprehensive treatment of
  dissipativity theory.

\begin{thm}
\label{thm:DI}
Suppose $X_j=X_j^\tp\in \R^{(n_{\xi}+n_w)\times (n_{\xi}+n_w)}$ for $j=1,2,\cdots, J$. Define $S_j:\R^{n_\xi}\times \R^{n_w}\to \R$ as
\begin{align}
\label{eq:supply}
S_j(\xi, w) \defeq \bmat{\xi \\ w}^\tp X_j \bmat{\xi \\ w}.
\end{align}
If there exists a positive semidefinite matrix $P\in \R^{n_{\xi}\times n_{\xi}}$ and non-negative scalars $\lambda_j$ such that
\begin{align}
\label{eq:lmi1}
\bmat{A^\tp P A-\rho^2 P & A^\tp P B \\ B^\tp P A & B^\tp P B}-\sum_{j=1}^J \lambda_j X_j\preceq 0,
\end{align}
then the a.s. exponential dissipation inequality~\eqref{eq:DI1} holds
for all sample paths of~\eqref{eq:sys1} with $V(\xi)\defeq\xi^\tp P
\xi$ and $S(\xi, w)\defeq\sum_{j=1}^J \lambda_j S_j(\xi, w)$. Further assuming that $\mathbb{E}S_j\le \Lambda_j$ for all sample paths of
\eqref{eq:sys1}, the following inequality always holds:
\begin{align}
\label{eq:keyineq}
\mathbb{E} V(\xi_{k+1})\le \rho^2 \mathbb{E}V(\xi_k)+\sum_{j=1}^J \lambda_j \Lambda_j.
\end{align}
\end{thm}
\begin{proof}
It is straightforward to verify
\begin{align*}
\begin{split}
V(\xi_{k+1})&=\xi_{k+1}^\tp P \xi_{k+1}\\
&=(A\xi_k+Bw_k)^\tp P (A\xi_k+Bw_k)\\
&=\bmat{\xi_k \\ w_k}^\tp \bmat{A^\tp P A & A^\tp P B \\ B^\tp P A & B^\tp P B} \bmat{\xi_k \\ w_k}.
\end{split}
\end{align*}
Hence, we can left- and right-multiply \eqref{eq:lmi1} by
$\bmat{\xi_k^\tp & w_k^\tp}$ and $\bmat{\xi_k^\tp & w_k^\tp}^\tp$ to
obtain the desired dissipation inequality. Since $\lambda_j$ is
non-negative, we take expectations of the dissipation inequality and
obtain \eqref{eq:keyineq}.
\end{proof}

\vspace{-0.15in}
If we fix $(A,B,X_j,\rho)$, the condition~\eqref{eq:lmi1} becomes an LMI with decision variables $P$ and $\lambda_j$.
For fixed $(A, B, X_j, \rho)$, the feasibility of \eqref{eq:lmi1} can be numerically tested using semidefinite programs. When applied to analyze stochastic optimization methods, the resulting LMI is typically small, and can also be solved analytically.

If one only wants to construct the dissipation inequality
\eqref{eq:DI1}, there is no need to enforce nonnegativity of
$\lambda_j$. However, we need $\lambda_j \ge 0$ to ensure that the
weighted supply rate $S=\sum_{j=1}^J \lambda_j S_j$ is useful in
convergence analysis.

We will use Theorem~\ref{thm:DI} to unify the analysis of SVRG and
Katyusha. The unified analysis follows four steps.
\begin{enumerate}
\item Rewrite the stochastic optimization methods in the form of a stochastic linear system \eqref{eq:sys1}.
\item Choose matrices $X_j$ in a way that the supply rate functions \eqref{eq:supply} satisfy certain desired properties.  
\item Solve the LMI \eqref{eq:lmi1} to obtain a dissipation inequality that directly yields the so-called one-iteration convergence result.
\item Apply some standard telescoping trick to convert the one-iteration convergence result into a rate bound for the analyzed method.
\end{enumerate}
Step 1 is straightforward. Step 4 has been routinized in the
literature. We will show how to perform Steps 2 and 3 for SVRG and
Katyusha.  Compared with \citet{BinHu2017}, the novelty of the present
paper is the development of several new stochastic supply rate
conditions that depend on the stochastic variance reduction mechanism.

For illustrative purposes, we first recall the LMI analysis for SG
using dissipativity theory.

\subsection{Demonstrative Example: Dissipativity for SG}
To gain some insight, we first rephrase the LMI-based analysis for SG in \citep{hu2017SGD} using dissipativity theory. 
SG uses the following iteration:
\begin{equation}
\label{eq:SG}
x_{k+1}=x_k-\eta \nabla f_{i_k}(x_k),
\end{equation}
where $i_k$ is sampled uniformly from $\{1, 2, \ldots, n\}$ at every
step.  Note that \eqref{eq:SG} is equivalent to
$x_{k+1}-x_\star=x_k-x_\star-\eta \nabla f_{i_k}(x_k)$. Hence we can
define $\xi_k=x_k-x_\star$, $w_k=\nabla f_{i_k}(x_k)=\nabla
f_{i_k}(\xi_k+x_\star)$, $A=I_p$, and $B=-\eta I_p$. Then the SG
iteration \eqref{eq:SG} is equivalent to the LTI model
\eqref{eq:sys1}.
Based on the properties of $f_i$ and $g$, we can choose the supply
rate functions $S_j(\xi, w)$ based on the following lemma.
\begin{lem} \label{lem:4}
Let $g$ be $L$-smooth and $\sigma$-strongly convex with
$\sigma>0$. Suppose $f_i$ is $L$-smooth and convex. Let $x_\star$ be
the point satisfying $\nabla g(x_\star)=~0$.  Define $X_1=\bar{X}_1
\otimes I_p$ and $X_2=\bar{X}_2 \otimes I_p$, where
\begin{align}
\bar{X}_1\defeq\bmat{ 2\sigma & -1 \\  -1 & 0 },\;\; \bar{X}_2\defeq\bmat{ 0 & -L \\  -L & 1}.
\end{align}
Consider $w_k=\nabla f_{i_k}(\xi_k+x_\star)$ where $i_k$ is sampled uniformly.
Define the supply rate  functions $S_1(\xi, w)$ and $S_2(\xi, w)$ using \eqref{eq:supply}.
Then the following supply rate conditions hold
\[
S_1 \le 0, \quad S_2\le \frac{2}{n}\sum_{i=1}^n \norm{\nabla f_i(x_\star)}^2.
\]
\end{lem}
\begin{proof}
The proof is given in \citet{hu2017SGD}. For completeness, we include
it in the supplementary material.
\end{proof}

Using Lemma~\ref{lem:4}, we can apply Theorem \ref{thm:DI} to
construct a dissipation inequality for SG. Setting $P=I_p$, the LMI
condition \eqref{eq:lmi1} becomes
\begin{align}
\label{eq:lmiSG}
\bmat{1-\rho^2-2\lambda_1 \sigma & -\eta+\lambda_1+\lambda_2 L \\ -\eta+\lambda_1+\lambda_2 L & \eta^2-\lambda_2}\otimes I_p \preceq 0.
\end{align}
Based on Remark \ref{rem:remark1}, we can show
$\mathbb{E}\norm{x_k-x_\star}^2 \le \rho^{2k} \norm{x_0-x_\star}^2
+\frac{2\lambda_2}{n(1-\rho^2)}\sum_{i=1}^n \norm{\nabla
  f_i(x_\star)}^2$ by finding non-negative $(\lambda_1, \lambda_2,
\rho^2)$ satisfying the above LMI. In fact, the choices
$\lambda_1=\eta-L\eta^2$, $\lambda_2=\eta^2$, and $\rho^2=1-2\lambda_1
\sigma$ suffice, since they make the left-hand side of \eqref{eq:lmiSG} zero.
We thus obtain the conclusion
\begin{multline*}
\mathbb{E}\norm{x_k-x_\star}^2 \le (1-2\sigma \eta+2\sigma L \eta^2)^k \norm{x_0-x_\star}^2 \\
+\frac{\eta}{\sigma(1-L\eta)}\left(\frac{1}{n}\sum_{i=1}^{n} \norm{\nabla f_i(x_\star)}^2\right),
\end{multline*}
which is the standard convergence result for SG
\citep[Theorem~2.1]{needell2014}. Since the supply rate $S_2$
continues to deliver energy into the system, the SG method with a
constant stepsize can only converge to a ball around the optimal
point.  Later we will see that SVRG-like methods adopt different
supply rate functions and eventually reduce their supply energy to
$0$, enabling linear convergence to the optimal point to be proved.

In this paper, we confine our scope to the case of constant learning
rate. For algorithms with time-varying learning rates, one may need to
adopt the dissipativity theory for linear time-varying (LTV)
systems. This theory requires time-varying Lyapunov functions and
infinite-dimensional LMIs. See \citet[Section~4.2]{BinHu2017} for
further discussions of this point.

\section{Dissipation Inequality for SVRG}

In this section, we present a unified LMI-based analysis for SVRG
using dissipativity theory. SVRG iterates as follows.  Let
$\tilde{x}^{0}\in \R^p$ be an arbitrary initial point.  For each epoch
$s=0, 1, \cdots$, we have $x_0^{s}=\tilde{x}^{s}$.  For each $s$, SVRG
performs the following steps for $k=0, 1, \ldots, m-1$:
\[
x_{k+1}^{s}=x_k^{s}-\eta \left(\nabla f_{i_k^{s}} (x_k^{s})-\nabla f_{i_k^{s}} (\tilde{x}^{s})+\nabla g(\tilde{x}^{s})\right),
\]
where $i_k^{s}$ is uniformly sampled from $\{1,2,\ldots, n\}$ in an
IID manner, and $m$ is a prescribed integer determining the epoch
length. A popular choice for $m$ is $m=2n$. At the end of each
epoch $s$, two typical options are available for updating
$\tilde{x}^{s+1}$:
\begin{itemize}
\item Option \rom{1}:  Set $\tilde{x}^{s+1}=x_m^{s}$;
\item Option \rom{2}\footnote{A similar variant with similar analysis
  is to choose $\tilde{x}^{s+1}$ by sampling uniformly from the
  iterates in the last epoch.}: Set $\tilde{x}^{s+1}=\frac{1}{m}
  \sum_{k=0}^{m-1} x_k^{s}$.
\end{itemize}

When analyzing SVRG, one typically needs to show that there exist $0\le \nu <1$ such that 
\begin{align}
\label{eq:RSVRG}
\mathbb{E} V(\tilde{x}^{s+1})\le \nu\, \mathbb{E} V(\tilde{x}^{s}),
\end{align}
where $V(\tilde{x}^s)$ is set to be either
$\norm{\tilde{x}^s-x_\star}^2$ or $g(\tilde{x}^s)-g(x_\star)$.  Since
\eqref{eq:RSVRG} needs to hold for all $s$, we can drop the
superscript $s$ in the so-called one-iteration analysis, and 
write each epoch of  SVRG in the form of the LTI
model~\eqref{eq:sys1}. Specifically, for a fixed $s$, we
have from the SVRG formula above that 
\begin{align}
\label{eq:sysSVRG}
x_{k+1}-x_\star=x_k-x_\star +B w_k, \quad k=0,1,\dotsc,m-1,
\end{align}
where $B=\bmat{-\eta I_p & -\eta I_p}$, and $w_k$ is given as
\begin{align}
\label{eq:SVRG_w}
w_k=\bmat{\nabla f_{i_k}(x_k)-\nabla f_{i_k}(x_\star) \\[2mm] \nabla f_{i_k}(x_\star)-\nabla f_{i_k}(\tilde{x})+\nabla g(\tilde{x})}.
\end{align}
With these choices of $w_k$ and $B$, we can set $\xi_k=x_k-x_\star$
and $A=I_p$ to recast SVRG in the linear model \eqref{eq:sys1}.  Next,
we will show how to construct supply rate functions for SVRG and apply
Theorem \ref{thm:DI} to obtain various rate bounds in the form of
\eqref{eq:RSVRG}. Our analysis recovers the existing bounds for SVRG,
and leads to more general characterizations of the convergence
properties of SVRG. We also give physical interpretations for the
convergence mechanism of SVRG.

\subsection{Warm-up: Dissipativity for SVRG with Option I}

Since we have already rewritten SVRG in the form of the linear model
\eqref{eq:sys1}, we can construct the dissipation inequality
efficiently for SVRG using semidefinite programs in
Theorem~\ref{thm:DI}.  As before there matrices are derived from
propoerties of $f_i$, $i=1,2,\dotsc,n$ and $g$, as we show now.
\begin{lem}
\label{lem:supply_SVRGI}
Suppose that $g$ is $L$-smooth and $\sigma$-strongly convex with
$\sigma>0$, and that $f_i$ is $L$-smooth and convex for
$i=1,2,\dotsc,n$. Suppose that $x_\star$ satisfies $\nabla
g(x_\star)=~0$.  Set $X_j=\bar{X}_j \otimes I_p$, where
$\bar{X}_j$, $j=1,2,3,4$ are defined as follows:
\begin{align}
\begin{split}
\bar{X}_1&=\bmat{ 0 & 0 & 0\\0 & 0 & 0\\ 0 & 0 & 1},\;\; \bar{X}_2=\bmat{ 2\sigma & -1 & -1\\-1 & 0 & 0\\ -1 & 0 & 0},\\
\bar{X}_3&=\bmat{ 0 & -L & 0\\ -L & 2 & 0 \\ 0 & 0 & 0}, \;\; \bar{X}_4=\bmat{0 & 0 & -1\\ 0 & 0 & 0 \\ -1 & 0 & 0}.
 \end{split}
\end{align}
Consider $\xi_k=x_k-x_\star$ and $w_k$ defined by \eqref{eq:SVRG_w}.
Suppose the supply rate $S_j$ is defined by \eqref{eq:supply} for
$j=1, 2, 3, 4$.  Then $\mathbb{E}S_1\le L^2
\mathbb{E}\norm{\tilde{x}-x_\star}^2$, $\mathbb{E} S_2\le 0$,
$\mathbb{E} S_3\le 0$, and $\mathbb{E} S_4=0$.
\end{lem}
\begin{proof}
It is straightforward to verify that the proposed supply rate
conditions are equivalent to standard inequalities (co-coercivity,
etc) in the literature.
\end{proof}

We will provide more guidelines for supply rate
constructions in the supplementary material.

We now apply Theorem~\ref{thm:DI} to perform LMI-based convergence
analysis for SVRG with Option \rom{1}.

\begin{cor}
Suppose $g$ is $\sigma$-strongly convex and $L$-smooth. In addition,
$f_i$ is assumed to be convex and $L$-smooth.  Let $0\le \rho^2<1$ be
given. If there exist nonnegative scalars $(\lambda_1, \lambda_2,
\lambda_3)$ and another scalar $\lambda_4$ (not necessarily
nonnegative) such that
\begin{align}
\label{eq:lmiSVRGI1}
\bmat{1-\rho^2-2\sigma \lambda_2& \lambda_2-\eta+L\lambda_3 & \lambda_2-\eta+\lambda_4\\ \lambda_2-\eta+L\lambda_3 & \eta^2-2\lambda_3 & \eta^2\\ \lambda_2-\eta+\lambda_4 & \eta^2 & \eta^2-\lambda_1 } \preceq 0,
\end{align}
then SVRG with Option \rom{1} satisfies
\begin{align}
\label{eq:SVRGcon0}
\mathbb{E} \|x_m-x_\star\|^2\le \left(\rho^{2m}+\frac{\lambda_1 L^2}{1-\rho^2}\right) \mathbb{E}\|x_0-x_\star\|^2.
\end{align}
\end{cor}
\begin{proof}
We choose the supply rate functions $S_j$ for $j=1,2,3,4$ as described
in Lemma \ref{lem:supply_SVRGI} and \eqref{eq:supply}.  Since $A=I_p$,
and $B=\bmat{-\eta I_p & -\eta I_p}$, we can set $P=I_p$ and show
\begin{align*}
\bmat{A^\tp P A-\rho^2 P & A^\tp P B \\ B^\tp P A & B^\tp P
  B}=\bmat{1-\rho^2 & -\eta & -\eta\\ -\eta & \eta^2 & \eta^2 \\ -\eta
  & \eta^2 & \eta^2}\otimes I_p.
\end{align*}
Thus the left-hand side of \eqref{eq:lmiSVRGI1} satisfies
\eqref{eq:lmi1} if $\lambda_4 \ge 0$. If $\lambda_4 <0$, we can
replace $X_4$ by $-X_4$ and $\lambda_4$ by $-\lambda_4$, and
\eqref{eq:lmi1} will hold with $\lambda_4$ now positive.  The
conclusion \eqref{eq:keyineq} is not affected by the change of sign,
since $\mathbb{E} S_4=0$, so we can set $\Lambda_4 =0$ in
\eqref{eq:keyineq}.  We have from the conclusion of
Theorem~\ref{thm:DI} that
$\mathbb{E}V(\xi_{k+1})=\mathbb{E}\|x_{k+1}-x_\star\|^2\le \rho^2
\mathbb{E} \|x_k-x_\star\|^2+\lambda_1
L^2\mathbb{E}\|x_0-x_\star\|^2$. We iterate this inequality over
$k=0,1,\dotsc,m-1$ to obtain the result.
\end{proof}
\vspace{-0.1in}

We can immediately show linear convergence of SVRG with Option \rom{1}  by choosing  $\lambda_1=2\eta^2$, $\lambda_2=\eta-L\eta^2$, $\lambda_3=\eta^2$, $\lambda_4=L\eta^2$, and $\rho^2=1-2\sigma(\eta-L\eta^2)$.  
Then \eqref{eq:lmiSVRGI1} becomes 
\begin{align*}
\bmat{0  & 0 & 0 \\ 0  & -\eta^2 & \eta^2 \\ 0 & \eta^2 & -\eta^2}\preceq 0,
\end{align*}
which is clearly true.
Hence \eqref{eq:RSVRG} holds with $V(\tilde{x}^s)=\norm{\tilde{x}^s-x_\star}^2$ and $\nu$ given by
\begin{align}
\nu= (1-2\eta\sigma(1-\eta L))^m+\frac{\eta L^2}{\sigma (1-\eta L)}. 
\end{align}
This bound slightly improves that of
\citep[Corollary~1]{tan2016barzilai}.  Other bounds under various
assumptions are discussed in the supplementary material.

\begin{rem}
The important physical insight is provided by the supply rate
condition $\mathbb{E}S_1\le L^2
\mathbb{E}\norm{\tilde{x}-x_\star}^2$. Although the supply rate $S_1$
is delivering energy into the system, the energy supplied is bounded
above by $L^2 \mathbb{E}\norm{\tilde{x}-x_\star}^2$, which diminishes
as $\tilde{x}$ approaches $x_\star$.  Eventually, the energy supplied
by $S_1$ cannot overcome dissipation.
\end{rem}

\subsection{LMI Analysis for SVRG with Option \rom{2}}

For SVRG with Option \rom{2}, we require the following supply rate functions.
\begin{lem}
\label{lem:supply_SVRGII}
Suppose that $g$ is $L$-smooth and $\sigma$-strongly convex with
$\sigma>0$, and that each $f_i$ is $L$-smooth and convex. Let
$x_\star$ be the point satisfying $\nabla g(x_\star)=~0$.  Set $X_j=\bar{X}_j \otimes I_p$, where
$\bar{X}_j$, $j=1,2,3$ are defined as:
\begin{align}
\begin{split}
\bar{X}_1&=\bmat{ 0 & 0 & 0\\0 & 1 & 0\\ 0 & 0 & 0},\;\; \bar{X}_2=\bmat{ 0 & 0 & 0\\0 & 0 & 0\\ 0 & 0 & 1},\\
\bar{X}_3&=\bmat{ 0 & -1 & -1\\ -1 & 0 & 0 \\ -1 & 0 & 0}.
 \end{split}
\end{align}
Define $\xi_k=x_k-x_\star$ and define $w_k$ as in \eqref{eq:SVRG_w}.
Suppose the supply rate $S_j$ is defined by \eqref{eq:supply} for
$j=1,2,3$.  Then the following supply rate conditions hold
\begin{subequations}
\begin{align}
\mathbb{E} S_1&\le 2L(\mathbb{E}g(x_k)-g(x_\star)),\\
\mathbb{E} S_2&\le 2L(\mathbb{E}g(\tilde{x})-g(x_\star)),\\
\mathbb{E} S_3&\le -\mathbb{E}g(x_k)+g(x_\star).
\end{align}
\end{subequations}
\end{lem}
\begin{proof}
These are also standard inequalities in the literature.  See
\citet{tan2016barzilai} and \citet[Lemma 6.4]{bubeck2015} for more
details. (Further discussions are provided in the supplementary
material.)
\end{proof}

\begin{cor}
  Suppose that $g$ is $\sigma$-strongly convex and $L$-smooth, and
  that each $f_i$ is convex and $L$-smooth.  If there exist
  non-negative scalars $\lambda_j$, $j=1,2,3$, such that
  $\lambda_3-L\lambda_1 > 0$ and
\begin{align}
\label{eq:lmiSVRGII1}
\bmat{0 & \lambda_3-\eta & \lambda_3-\eta\\ \lambda_3-\eta & \eta^2-\lambda_1 & \eta^2\\ \lambda_3-\eta & \eta^2 & \eta^2-\lambda_2 } \preceq 0,
\end{align}
then SVRG with Option \rom{2} satisfies
\begin{multline}
\label{eq:SVRGcon1}
\mathbb{E} g\left(\frac{1}{m}\sum_{k=0}^{m-1} x_k\right)-g(x_\star)\\ \le
\left(\frac{\sigma^{-1}+mL\lambda_2}{(\lambda_3-L\lambda_1)m}\right)(\mathbb{E}g(\tilde{x})-g(x_\star)).
\end{multline}
\end{cor}
\begin{proof}
Recall that $A=I_p$, and $B=\bmat{-\eta I_p & -\eta I_p}$ for
the state-space representation of SVRG. Let $S_1$, $S_2$, and
$S_3$ be the supply rate functions defined from $X_1$, $X_2$, and
$X_3$ of Lemma~\ref{lem:supply_SVRGII} via \eqref{eq:supply}.
Setting $P=I_p$ and $\rho=1$, 
the left-hand side of the LMI \eqref{eq:lmi1} becomes
\begin{align*}
\bmat{0 & \lambda_3-\eta & \lambda_3-\eta\\ \lambda_3-\eta & \eta^2-\lambda_1 & \eta^2\\ \lambda_3-\eta & \eta^2 & \eta^2-\lambda_2 }\otimes I_p.
\end{align*}
Since \eqref{eq:lmiSVRGII1} holds, can apply Theorem~\ref{thm:DI} to
show that
\begin{multline*}
\mathbb{E} \|x_{k+1}-x_\star\|^2\le  \mathbb{E}\|x_k-x_\star\|^2\\
-(2\lambda_3-2L\lambda_1)(\mathbb{E}g(x_k)-g(x_\star))+2L\lambda_2(\mathbb{E}g(\tilde{x})-g(x_\star)).
\end{multline*}
We can sum the above inequality from $k=0$ to $m-1$ and show that 
\begin{multline}
\label{eq:SVRGII1}
(2\lambda_3-2L\lambda_1)\sum_{k=0}^{m-1}\left(\mathbb{E}g(x_k)-g(x_\star)\right)\\ \le
 \mathbb{E}\|x_0-x_\star\|^2+2mL\lambda_2 \mathbb{E} (g(\tilde{x})-g(x_\star)).
\end{multline}
By convexity of $g$, we have 
\begin{align*}
g\left(\frac{1}{m}\sum_{k=0}^{m-1} x_k \right)\le \frac{1}{m}\sum_{k=0}^{m-1} g(x_k).
\end{align*}
Since $g$ is $\sigma$-strongly convex, we also have
$\|x_0-x_\star\|^2\le \frac{2}{\sigma}(\mathbb{E}g(x_0)-g(x_\star))$.
By substituting these inequalities into \eqref{eq:SVRGII1}, and using
the assumption $\lambda_3-L\lambda_1 > 0$, we obtain the result.
\end{proof}

We can recover the standard rate result for SVRG by choosing
$\lambda_1=\lambda_2=2\eta^2$, and $\lambda_3=\eta$. We have
$\lambda_3-L\lambda_1=\eta-L\eta^2 \ge 0$ for $\eta \le \frac{1}{L}$,
and \eqref{eq:lmiSVRGII1} becomes
\begin{align*}
\bmat{0  & 0 & 0 \\ 0  & -\eta^2 & \eta^2 \\ 0 & \eta^2 & -\eta^2}\preceq 0,
\end{align*}
which is clearly true.  Additionally, we have
\begin{align}
\frac{\sigma^{-1}+mL\lambda_2}{(\lambda_3-L\lambda_1)m}=\frac{1}{m\sigma\eta(1-2L\eta)}+\frac{2L\eta}{1-2L\eta},
\end{align}
which is exactly the rate in \citep[Theorem~1]{johnson2013}. This
result states that the iteration complexity of SVRG with Option
\rom{2} is
$\mathcal{O}\left((\frac{L}{\sigma}+n)\log(\frac{1}{\epsilon})\right)$
if we choose $m=\frac{20L}{\sigma}$.

\begin{rem}
 Some important physical insight is provided by the supply
  rate condition $\mathbb{E}S_2\le 2L
  \mathbb{E}(g(\tilde{x})-g(x_\star))$.  As $\tilde{x}$ approaches
  $x_\star$, the energy supplied by $S_2$ drops and is unable to
  overcome dissipation, leading to convergence.  One may add more
  supply rate functions and improve the convergence guarantees by some
  constant factor. In principle, the introduction of more supply rate
  functions may reduce the conservatism in the analysis. Other choices of $\lambda_j$ may also change the iteration complexity by a constant factor. In addition, new choices of $m$ may require different choices of $\lambda_j$. Note that LMI \eqref{eq:lmi1}  in Theorem \ref{thm:DI} can be implemented and solved numerically, leading to numerical clues for how to construct $P$ and $\lambda_j$ for proving rate results. Therefore, our proposed
  LMI provides an efficient tool for constructing bounds of the form
  \eqref{eq:SVRGcon1}. 
\end{rem}

\section{Dissipativity Theory for Katyusha}

Katyusha  solves the following problem:
\begin{align}
  \nonumber
\min_{x\in \R^p} \,\, F(x) & \defeq f(x)+\psi(x)\\
\label{eq:minP1}
&= \frac{1}{n} \sum_{i=1}^n f_i(x)+\psi(x),
\end{align}
where $\psi$ is $\sigma$-strongly convex and possibly nonsmooth, while
each $f_i$, $i=1,2,\dotsc,n$ is $L$-smooth and convex.

For each epoch $s=0, 1, \cdots$, we have
$y_0^{s}=z_0^s=\tilde{x}^{s}$.  For any fixed $s$, and positive
parameters $\tau_1$, $\tau_2$, and $\alpha$, Katyusha applies the
following iteration for $k=0, 1, \ldots, m-1$:
\begin{subequations}
\begin{align}
\label{eq:Katyusha}
x_{k+1}^s&=\tau_1 z_k^s+\tau_2 \tilde{x}^s+(1-\tau_1-\tau_2) y_k^s,\\
v_k^s&= \nabla f_{i_k^{s}} (x_{k+1}^{s})-\nabla f_{i_k^{s}} (\tilde{x}^{s})+\nabla f(\tilde{x}^{s}),\\
z_{k+1}^s&=\argmin_z \left\{\frac{1}{2\alpha}\norm{z-z_k^s}^2+(v_k^s)^\tp z+\psi(z)\right\},\\
y_{k+1}^s&=\argmin_y\left\{\frac{3L}{2}\norm{y-x_{k+1}^s}^2+(v_k^s)^\tp y+\psi(y)\right\},
\end{align}
\end{subequations}
where $i_k^{s}$ is uniformly sampled from $\{1,2,\ldots, n\}$ in an
i.i.d. manner, and $m$ is a prescribed integer determining the length
of the epoch. (A popular choice is $m=2n$.) At the end of each epoch
$s$, we set
\begin{align}
\tilde{x}^{s+1}=\left(\sum_{j=0}^{m-1}(1+\sigma\alpha)^j\right)^{-1}\left(\sum_{j=0}^{m-1}(1+\alpha\sigma)^j y_{j+1}^s\right).
\end{align}

\citet{katyusha2016} shows that the iteration complexity for Katyusha
is
$\mathcal{O}\left( \left(\sqrt{\frac{Ln}{\sigma}}+n \right) \log(\frac{1}{\epsilon})\right)$
if one chooses $\tau_2=\frac{1}{2}$, $\tau_1=\min
\{\sqrt{\frac{m\sigma}{3L}},\frac{1}{2}\}$, $\alpha=\frac{1}{3\tau_1
  L}$, and $m=2n$.  The key of the proof is the coupling lemma
\citep[Lemma 3.7]{katyusha2016}, which states the following holds for
Katyusha with $\tau_1\le \frac{1}{3\alpha L}$ and
$\tau_2=\frac{1}{2}$:
\begin{align}
  \nonumber
  &\frac{1+\alpha \sigma}{2}\mathbb{E}\norm{z_{k+1}-x_\star}^2+\frac{\alpha}{\tau_1}\left(\mathbb{E}F(y_{k+1})-F(x_\star)\right)\\
  \nonumber
&-\frac{1}{2}\mathbb{E}\norm{z_k-x_\star}^2-\frac{\alpha(1-\tau_1-\tau_2)}{\tau_1}\left(\mathbb{E}F(y_k)-F(x_\star)\right)\\
\label{eq:coupling}
&\le \frac{\alpha \tau_2}{\tau_1}\left(\mathbb{E}F(\tilde{x})-F_\star\right).
\end{align}
We analyze a single epoch, dropping the superscript  $s$ to
simplify the notation.  As stated in \citep[Section
  3.2]{katyusha2016}, once the above one-iteration convergence result
is established, a telescoping trick can be applied to  show
the improved iteration complexity of Katyusha.
We show how to provide a general proof for \eqref{eq:coupling} using
dissipativity, with Theorem~\ref{thm:DI} again being our main
technical tool. 

\subsection{Katyusha as a Stochastic System}
At a given epoch $s$ (subscript dropped), a single ``inner'' iteration
of Katyusha can be written as follows:
\begin{subequations}
\label{eq:Kat1}
\begin{align}
x_{k+1}&=\tau_1 z_k+\tau_2 \tilde{x}+(1-\tau_1-\tau_2) y_k,\\
v_k&=\nabla f_{i_k}(x_{k+1})-\nabla f_{i_k}(\tilde{x})+\nabla f(\tilde{x}),\\
z_{k+1}&=z_k-\alpha v_k -\alpha g_k,\\
y_{k+1}&= x_{k+1}- \zeta v_k -\zeta h_k,
\end{align}
\end{subequations}
where $g_k$ is some subgradient of $\psi$ evaluated at $z_{k+1}$, and
$h_k$ is some subgradient of $\psi$ evaluated at $y_{k+1}$. We can set
$\zeta=\frac{1}{3L}$ to recover the standard Katyusha iteration
\ref{eq:Katyusha}.

We can rewrite \eqref{eq:Kat1} as 
\begin{align*}
\begin{split}
\bmat{z_{k+1}-x_\star \\ y_{k+1}-x_\star \\ \tilde{x}-x_\star}=A \bmat{z_k-x_\star \\ y_k-x_\star \\ \tilde{x}-x_\star}+B\bmat{v_k \\ g_k \\ h_k },
\end{split}
\end{align*}
where $A=\bar{A} \otimes I_p$ and $B = \bar{B} \otimes I_p$, and
$\bar{A}$ and $\bar{B}$ are defined as follows:
\[
\bar{A}=\bmat{1 & 0 & 0\\ \tau_1 & 1-\tau_1-\tau_2 & \tau_2 \\ 0 & 0 & 1}, \;\;
 \bar{B}=\bmat{-\alpha  & -\alpha & 0 \\ -\zeta & 0 & -\zeta \\ 0  & 0 & 0}.
\]
Based on the iteration above, it is straightforward to check that
Katyusha \eqref{eq:Kat1} is equivalent to the stochastic linear system
\eqref{eq:sys1} with
\begin{equation}
\label{eq:Katss}
\xi_k=\bmat{z_k-x_\star \\ y_k-x_\star \\ \tilde{x}-x_\star}, \quad w_k=\bmat{v_k  \\ g_k \\ h_k}.
\end{equation}

\subsection{Supply Rate Functions for Katyusha}
Katyusha extracts energy out of the system much faster than SVRG, as
can be shown by the use of more advanced supply rate functions.

\begin{lem}
\label{lem:SupplyKatyusha}
Let $\psi$ be $\sigma$-strongly convex with $\sigma>0$. Suppose $f_i$
is $L$-smooth and convex. Let $x_\star$ be the optimal point of $F$.
Define $X_1=\bar{X}_1 \otimes I_p$, where $\bar{X}_1$ is the following
sum of four matrices:
\begin{multline}
\bar{X}_1=-\bmat{ -\frac{\sigma\tau_1}{2} & 0 & 0 & \frac{\tau_1(\alpha\sigma+1)}{2} &  \frac{\tau_1(\alpha\sigma+1)}{2} & 0\\
0 & 0 & 0 & 0 & 0 & 0\\
0 & 0 & 0 & 0 & 0 & 0\\
\frac{\tau_1(\alpha\sigma+1)}{2} & 0 & 0 & 0 & 0 & 0\\
\frac{\tau_1(\alpha\sigma+1)}{2} & 0 & 0 & 0 & 0 & 0\\
0 & 0 & 0 & 0 & 0 & 0
}\\
+\left(\zeta-\frac{\alpha\tau_1}{2}-\frac{L\zeta^2(1+\tau_2)}{2\tau_2}\right)\bmat{ 0 & 0 & 0 & 0 &  0 & 0\\
0 & 0 & 0 & 0 & 0 & 0\\
0 & 0 & 0 & 0 & 0 & 0\\
0 & 0 & 0 & 1 & 0 & 1\\
0 & 0 & 0 & 0 & 0 & 0\\
0 & 0 & 0 & 1 & 0 & 1
}\\
+\frac{\alpha\tau_1(\alpha\sigma+1)}{2}\bmat{ 0 & 0 & 0 & 0 &  0 & 0\\
0 & 0 & 0 & 0 & 0 & 0\\
0 & 0 & 0 & 0 & 0 & 0\\
0 & 0 & 0 & 1 & 1 & 0\\
0 & 0 & 0 & 1 & 1 & 0\\
0 & 0 & 0 & 0 & 0 & 0
}\\
+\frac{\alpha\tau_1}{2}\bmat{ 0 & 0 & 0 & 0 &  0 & 0\\
0 & 0 & 0 & 0 & 0 & 0\\
0 & 0 & 0 & 0 & 0 & 0\\
0 & 0 & 0 & 0 & 0 & 0\\
0 & 0 & 0 & 0 & 1 & -1\\
0 & 0 & 0 & 0 & -1 & 1
}.
\end{multline}
Consider $\xi_k$ and $w_k$ defined by \eqref{eq:Katss}.
Suppose the supply rate $S_j$ is defined by \eqref{eq:supply} for $j=1$. 
Then the following supply rate condition holds for Katyusha
\begin{multline}
\label{eq:KatyushaSupply}
\mathbb{E}S_1(\xi_k, w_k) \le (1-\tau_1-\tau_2) (\mathbb{E}F(y_k)-F(x_\star))\\-(\mathbb{E}F(y_{k+1})-F(x_\star))+\tau_2 (\mathbb{E}F(\tilde{x})-F(x_\star)).
\end{multline}
\end{lem}
\begin{proof}
The proof is based on the strong-convexity of $\psi$, and the smoothness and convexity of $f_i$. The detailed proof is presented in the supplementary material.
\end{proof}

The physical interpretation for the above supply rate is as
follows. There is some hidden energy in the system that takes the form
of $F(y_k)- F(x_\star)$. There is also some initial energy in the form
of $F(\tilde{x})-F(x_\star)$. The above supply rate condition states
that the delivered energy is bounded by a weighted decrease of the
hidden energy plus some amount of the initial energy. Such a supply
rate can efficiently extract energy out of the systems due to its
coupling with the hidden energy and the initial energy.  The supply
rate construction in Lemma~\ref{lem:SupplyKatyusha} is quite similar
to the supply rate construction for Nesterov's accelerated method
\citep{BinHu2017}. From a physical viewpoint, the essential property
of momentum terms can extract the hidden energy out of the system in a
more efficient way.

\begin{rem}
Although the supply rate in Lemma \ref{lem:SupplyKatyusha} is complicated, there are some general guidelines for constructing and choosing supply rates. We discuss these guidelines in the supplementary materials.
\end{rem}

\subsection{Analysis of Katyusha Using Dissipativity}

Using the supply rate function in Lemma \ref{lem:SupplyKatyusha}, we can immediately recover the one-iteration result \eqref{eq:coupling} as follows. Suppose $\tau_2=\frac{1}{2}$ and $\zeta=\frac{1}{3L}$. We choose $\rho^2=\frac{1}{1+\alpha\sigma}$, $\lambda_1=\frac{\alpha}{\tau_1}$, and
\begin{align}
\label{eq:Pdef}
P=\frac{1+\alpha \sigma}{2}\bmat{1 & 0 & 0\\ 0 & 0 & 0 \\ 0 & 0 & 0}\otimes I_p.
\end{align}
Then the left-hand side of the LMI condition \eqref{eq:lmi1} becomes
\begin{multline*}
\frac{\alpha}{2}\left(\alpha-\frac{1}{3L\tau_1}\right) \bmat{ 0 & 0 & 0 & 0 &  0 & 0\\
0 & 0 & 0 & 0 & 0 & 0\\
0 & 0 & 0 & 0 & 0 & 0\\
0 & 0 & 0 & 1 & 0 & 1\\
0 & 0 & 0 & 0 & 0 & 0\\
0 & 0 & 0 & 1 & 0 & 1}
\otimes I_p\\
+ \frac{\alpha^2}{2}\bmat{ 0 & 0 & 0 & 0 &  0 & 0\\
0 & 0 & 0 & 0 & 0 & 0\\
0 & 0 & 0 & 0 & 0 & 0\\
0 & 0 & 0 & 0 & 0 & 0\\
0 & 0 & 0 & 0 & -1 & 1\\
0 & 0 & 0 & 0 & 1 & -1
} \otimes I_p,
\end{multline*}
which is clearly negative semidefinite when $\tau_1\le \frac{1}{3\alpha L}$.

Therefore, we can apply Theorem~\ref{thm:DI} to prove that 
\begin{multline*}
\frac{1+\alpha \sigma}{2}\mathbb{E}\norm{z_{k+1}-x_\star}^2-\frac12 \mathbb{E}\norm{z_k-x_\star}^2\\
\le \frac{\alpha}{\tau_1} \mathbb{E}S(\xi_k, w_k).
\end{multline*}
From the supply rate condition \eqref{eq:KatyushaSupply}, we
immediately recover the one-iteration analysis result
\eqref{eq:coupling}, which can be easily transferred into the
iteration complexity result by applying the telescoping trick in
\citet{katyusha2016}.

We emphasize that Lemma~\ref{lem:SupplyKatyusha} works for general
choices of $\zeta$ and $\tau_2$. Due to the generality of
Lemma~\ref{lem:SupplyKatyusha}, our LMI approach can be used to
generalize \citep[Lemma~3.7]{katyusha2016} for many more choices of
$(\tau_1, \tau_2)$ . This could lead to other choices of $(\tau_1,
\tau_2, \alpha, \zeta)$, which yields the same accelerated iteration
complexity. However, those choices of parameters will at most improve
the iteration complexity by a constant factor.  For example, consider
$\zeta=\frac{1}{3L}$ and any $\tau_2\ge\frac{1}{5}$. We can still
choose $\rho^2=\frac{1}{1+\alpha\sigma}$,
$\lambda_1=\frac{\alpha}{\tau_1}$, and $P$ as defined in
\eqref{eq:Pdef} to prove \eqref{eq:coupling}. In this case, the
left-hand side of the LMI condition \eqref{eq:lmi1} becomes
\begin{multline*}
\frac{\alpha}{2}\left(\alpha-\frac{5\tau_2-1}{9L\tau_1\tau_2}\right) \bmat{ 0 & 0 & 0 & 0 &  0 & 0\\
0 & 0 & 0 & 0 & 0 & 0\\
0 & 0 & 0 & 0 & 0 & 0\\
0 & 0 & 0 & 1 & 0 & 1\\
0 & 0 & 0 & 0 & 0 & 0\\
0 & 0 & 0 & 1 & 0 & 1}
\otimes I_p\\
+ \frac{\alpha^2}{2}\bmat{ 0 & 0 & 0 & 0 &  0 & 0\\
0 & 0 & 0 & 0 & 0 & 0\\
0 & 0 & 0 & 0 & 0 & 0\\
0 & 0 & 0 & 0 & 0 & 0\\
0 & 0 & 0 & 0 & -1 & 1\\
0 & 0 & 0 & 0 & 1 & -1
} \otimes I_p,
\end{multline*}
which is clearly negative semidefinite when $\tau_1\le \frac{5\tau_2-1}{9\alpha L\tau_2}$. Therefore, the one-iteration convergence result \eqref{eq:coupling} holds for any $\frac{1}{5}\le \tau_2 <1$ and $\tau_1\le \min\{\frac{5\tau_2-1}{9\alpha L\tau_2}, 1-\tau_2\}$.
This generalizes the coupling lemma \citep[Lemma 3.7]{katyusha2016} to more general choices of $(\tau_1, \tau_2)$.
Based on this, one can use the telescoping trick to show Katyusha with $\tau_2 \neq \frac{1}{2}$ can also achieve the iteration complexity of $\mathcal{O}\left((\sqrt{\frac{Ln}{\sigma}}+n)\log(\frac{1}{\epsilon})\right)$.

\section{Future Work}

We plan to use our techniques to
study optimal tuning of Katyusha X \citep{allen2018katyusha} for ERM
problems in which the component functions $f_i$ are not individually
convex. In addition, we are interested in investigating how to  accelerate other recently-developed stochastic methods such as SARAH \citep{nguyen17b} using our LMI approach. It is also important to extend our control framework for understanding other accelerating mechanism such as catalyst \citep{lin2015}.

Notice that deterministic continuous-time algorithms have also been understood as dissipative dynamical systems  \citep{attouch2000heavy, alvarez2002second, BinHu2017}. It is possible that one can modify the proposed framework to study stochastic continuous-time dynamics. This is another important future direction.\newpage

\section*{Acknowledgments}

Bin Hu and Laurent Lessard are supported by the National Science Foundation (NSF) under Grants No. 1656951 and 1750162.  Bin Hu and Laurent Lessard also
acknowledge support from the Wisconsin Institute for Discovery, the
College of Engineering, and the Department of Electrical and Computer
Engineering at the University of Wisconsin--Madison. Stephen Wright
was supported by NSF Awards IIS-1447449, 1628384, 1634597, and
1740707; AFOSR Award FA9550-13-1-0138; Subcontracts 3F-30222 and
8F-30039 from Argonne National Laboratory; and DARPA Award
N660011824020.

\bibliography{dissipativity_stochastic}
\bibliographystyle{icml2018}

\clearpage
\setcounter{section}{0}
\renewcommand{\thesection}{\Alph{section}}

\onecolumn
\begin{center}
{\Large\bf Supplementary Material}
\end{center}

The underlying probability space for the sampling index $i_k$ is
denoted by $(\Omega, \mathcal{F}, \mathbb{P})$. We denote by
$\mathcal{F}_k$ the $\sigma$-algebra generated by $(i_0, i_1, \ldots,
i_k)$. Clearly, $i_k$ is $\mathcal{F}_k$-adapted and we obtain a
filtered probability space $(\Omega, \mathcal{F}, \{\mathcal{F}_k\},
\mathbb{P})$ on which the stochastic optimization method is defined.

\section{Proof of Lemma \ref{lem:4}}
The proof is straightforward and included here only for
completeness. Note that $x_k$ does not depend on $i_k$, so we have
$\mathbb{E}\left[(x_k-x_\star)^\tp \nabla f_{i_k}(x_k)
  \middle|\ \mathcal{F}_{k-1}\right]=(x_k-x_\star)^\tp \nabla g(x_k)$.
If $g$ is $\sigma$-strongly convex, we directly have
\begin{align*}
\mathbb{E}\left[ \bmat{ x_k-x_\star \\ \nabla f_{i_k}(x_k) }^\tp \left(\bmat{ 2\sigma & -1\\ -1 &  0} \otimes I_p\right) \bmat{ x_k-x_\star \\ \nabla f_{i_k}(x_k) } \right]=\mathbb{E}\left[ \bmat{ x_k-x_\star \\ \nabla g(x_k) }^\tp \left(\bmat{ 2\sigma & -1\\ -1 &  0} \otimes I_p\right) \bmat{ x_k-x_\star \\ \nabla g(x_k) } \right] \le 0.
\end{align*}
Next, if $f_i$ is convex and $L$-smooth, the co-coercivity property implies
\begin{align*}
\bmat {x_k-x_\star \\ \nabla f_i(x_k)-\nabla f_i(x_\star)}^\tp
\left(\bmat{ 0 &  -L \\ -L & 2}\otimes I_p\right)
\bmat{ x_k-x_\star \\ \nabla f_i(x_k)-\nabla f_i(x_\star)}\le 0.
\end{align*}
Therefore, we have
\begin{align*}
&\mathbb{E} \left(\bmat{ x_k-x_\star \\ \nabla f_{i_k}(x_k) }^\tp
\left(\bmat{ 0 &  -L \\ -L & 1}\otimes I_p\right)
\bmat{ x_k-x_\star \\ \nabla f_{i_k}(x_k) }\,\,\middle|\ \, \mathcal{F}_{k-1}\right) \\
&= \frac{1}{n} \sum_{i=1}^n \bmat{ x_k-x_\star \\ \nabla  f_i(x_k)}^\tp \left(\bmat{ 0 &  -L \\ -L & 0}\otimes I_p\right)
\bmat{ x_k-x_\star \\ \nabla  f_i(x_k)}+ \frac{1}{n} \sum_{i=1}^n \norm{\nabla f_i(x_k)}^2\\
&\le -\frac{2}{n}\sum_{i=1}^n \norm{\nabla f_i(x_k)-\nabla f_i(x_\star)}^2+\frac{1}{n} \sum_{i=1}^n \norm{\nabla f_i(x_k)}^2\\
&\le \frac{2}{n}\sum_{i=1}^n \norm{\nabla f_i(x_\star)}^2.
\end{align*}
Taking the expectation of the above inequality leads to the desired conclusion.

\section{Proof of Lemma \ref{lem:supply_SVRGI} and Lemma \ref{lem:supply_SVRGII}}

We summarize some existing function inequalities that can be used to directly show Lemma \ref{lem:supply_SVRGI} and Lemma \ref{lem:supply_SVRGII}.

\begin{lemm}
\label{lem:SQClemma}
Assume $\nabla g(x_\star)=0$. Suppose $i_k$ is uniformly sampled from
$\{1,\ldots, n\}$ in an i.i.d. manner. Let $\{x_k:k=0,1,\ldots\}$ be
an $\mathcal{F}_{n}$- predictable process whose sample path satisfies
$x_k\in \R^p$ almost surely. In addition, $r_k=\nabla
f_{i_k}(x_k)-\nabla f_{i_k}(x_\star)$ and $u_k=\nabla
f_{i_k}(x_\star)-\nabla f_{i_k} (\tilde{x})+\nabla g(\tilde{x})$, where
$\tilde{x}$ is $\mathcal{F}_0$-measurable.
\begin{enumerate}
\item The following always holds due to the uniform sampling strategy:
\begin{align}\tag{S1}
\label{eq:SQC1}
\mathbb{E}\left[(x_k-x_\star)^\tp (\nabla f_{i_k}(x_\star)-\nabla f_{i_k}(\tilde{x})+\nabla g(\tilde{x}))\right]=0.
\end{align}

\item If $f_i$ is $L$-smooth, then 
\begin{align}\tag{S2}
\label{eq:SQC2}
\mathbb{E} \|\nabla f_{i_k}(x_\star)-\nabla f_{i_k} (\tilde{x})+\nabla g(\tilde{x})\|^2 \le L^2\mathbb{E}\|\tilde{x}-x_\star\|^2.
\end{align}

\item If $f_i$ is convex and $L$-smooth, then
\begin{align}\tag{S3}
\label{eq:SQC5}
&\mathbb{E}\|\nabla f_{i_k}(x_k)-\nabla f_{i_k}(x_\star)\|^2 \le 2L(\mathbb{E}g(x_k)-g(x_\star)),\\
\tag{S4}\label{eq:SQC5b}
&\mathbb{E}\|\nabla f_{i_k}(x_\star)-\nabla f_{i_k}(\tilde{x})+\nabla g(\tilde{x})\|^2 \le 2L(\mathbb{E}g(\tilde{x})-g(x_\star)).
\end{align}

\item 
 The following inequality holds
\begin{align}\tag{S5}
\label{eq:SQC3}
\mathbb{E}\left[ \bmat{ x_k-x_\star \\ r_k }^\tp (M\otimes I_p)  \bmat{ x_k-x_\star \\ r_k }\right]\le 0,
\end{align}
where $M$ is computed according to the assumption on $f_i$ as follows
\begin{align}\tag{S6}
  \label{eq:Mdef}
  M \defeq  \left\{
    \begin{array}{ll}
      \bmat{2\sigma L & -(\sigma+L) \\ -(\sigma+L) & 2} & \mbox{if } f_i \mbox{is } L\mbox{-smooth and } \sigma\mbox{-strongly convex},\\[3mm]
      \bmat{0 & -L \\ -L & 2} & \mbox{if }  f_i \mbox{is } L\mbox{-smooth and convex},\\[3mm]
     \bmat{-2L^2 & 0 \\ 0 & 2} & \mbox{if } f_i \mbox{ is } L\mbox{-smooth.} 
    \end{array}
  \right.
\end{align}

\item If $g$ is $\sigma$-strongly convex, we have
\begin{align}\tag{S7}
\label{eq:SQC4}
\mathbb{E}\left[ \bmat{ x_k-x_\star \\ r_k }^\tp \left(\bmat{ 2\sigma & -1\\ -1 &  0} \otimes I_p\right) \bmat{ x_k-x_\star \\ r_k } \right] \le 0.
\end{align}

\item If $g$ is convex, then
\begin{align}\tag{S8}
\label{eq:SQC6}
\mathbb{E} \left[(x_k-x_\star)^\tp (\nabla f_{i_k}(x_k)-\nabla f_{i_k} (\tilde{x})+\nabla g(\tilde{x}))\right]\ge \mathbb{E}g(x_k)-g(x_\star).
\end{align}

\item  If $g$ is $\sigma$-strongly convex, then
\begin{align}\tag{S9}
\label{eq:g_strongconvex}
\mathbb{E}\|\tilde{x}-x_\star\|^2\le \frac{2}{\sigma}\left(\mathbb{E}g(\tilde{x})-g(x_\star)\right).
\end{align}

\end{enumerate}
\end{lemm}
\begin{proof}
The proof is standard and based on the fact that $i_k$ and $x_k$ are independent. For example, we have
\begin{align*}
\mathbb{E}\left[(x_k-x_\star)^\tp (\nabla f_{i_k}(x_\star)-\nabla f_{i_k}(\tilde{x})+\nabla g(\tilde{x}))
    \middle|\ \mathcal{F}_{k-1}\right]=(x_k-x_\star)^\tp \nabla g(x_\star)=0,
\end{align*}
which directly leads to Statement 1. Note that $\mathbb{E}\left[\nabla
  f_{i_k}(x_\star)-\nabla f_{i_k}(\tilde{x}) \right]=-\mathbb{E}\nabla
g(\tilde{x})$. Hence, we have
\begin{align*}
\mathbb{E} \|\nabla f_{i_k}(x_\star)-\nabla f_{i_k} (\tilde{x})+\nabla g(\tilde{x})\|^2 \le\mathbb{E} \|\nabla f_{i_k}(x_\star)-\nabla f_{i_k} (\tilde{x})\|^2\le L^2\mathbb{E}\|\tilde{x}-x_\star\|^2,
\end{align*}
which proves Statement 2. The other statements follow from taking
expectations of well known function inequalities.
\end{proof}

The proofs of Lemma \ref{lem:supply_SVRGI} and Lemma \ref{lem:supply_SVRGII} directly follow from the lemma
above.

\section{Further Discussion on SVRG}

One can automate the convergence analysis for SVRG under various
assumptions on $f_i$.  For example, consider the analysis of SVRG with
Option I. If $f_i$ is assumed only to be $L$-smooth, we can modify
$\bar{X}_3$ in Lemma \ref{lem:supply_SVRGI} as
\begin{align*}
\bar{X}_3=\bmat{ -2L^2 & 0 & 0\\ 0 & 2 & 0 \\ 0 & 0 & 0}.
\end{align*}
We still assume that $g$ is $L$-smooth and $\sigma$-strongly convex,
so we choose $\bar{X}_1$, $\bar{X}_2$, and $\bar{X}_4$ as in Lemma
5. For these choices, it is still true that $\mathbb{E}S_1\le L^2
\mathbb{E}\norm{\tilde{x}-x_\star}^2$, $\mathbb{E} S_2\le 0$,
$\mathbb{E} S_3\le 0$, and $\mathbb{E} S_4=0$. The usual analysis
route leads to the following bound:
\begin{align*}
\mathbb{E}\|x_m-x_\star\|^2\le \left((1-2\sigma\eta+2 L^2\eta^2)^m+\frac{\eta L^2}{\sigma -\eta L^2} \right) \mathbb{E}\|x_0-x_\star\|^2.
\end{align*}
This example demonstrates that one can modify the supply rate
functions to reflect various assumptions on the cost functions. For
SVRG with Option II, one can perform similar LMI analysis when the
assumptions on $f_i$ are changed.

\section{Proof of Lemma \ref{lem:SupplyKatyusha}}

We first set
\begin{align}\tag{S10}
 q_k = \bmat{\tau_1& 1-\tau_1-\tau_2 & \tau_2}\bmat{z_k \\ y_k \\ \tilde{x}}.
 \end{align}

From the definition of Katyusha, we have $\mathbb{E} v_k=\mathbb{E} \nabla f(q_k)$. Since $f$ is $L$-smooth and convex, it is straightforward to verify the following:
\begin{subequations}
\label{eq:ineq_first1}
\begin{align}\tag{S11}
\mathbb{E}f(q_k)-\mathbb{E}f(y_k)&\le \mathbb{E}\nabla f(q_k)^\tp (q_k-y_k) = \mathbb{E}\big[ \mathbb{E}[v_k^\tp (q_k-y_k)| \mathcal{F}_{i_{k-1}}]\big]=\mathbb{E} v_k^\tp (q_k-y_k),\\
\tag{S12}
\mathbb{E} f(q_k)-\mathbb{E}f(x_\star)&\le \mathbb{E}\nabla f(q_k)^\tp (q_k-x_\star) = \mathbb{E}v_k^\tp (q_k-x_\star),\\
\mathbb{E}f(y_{k+1})-\mathbb{E}f(q_k)&\le \mathbb{E}\left[ \nabla f(q_k)^\tp (y_{k+1}-q_k)+\frac{L}{2}\norm{y_{k+1}-q_k}^2\right]\notag\\
&=\mathbb{E}\left[(\nabla f(q_k)-v_k)^\tp(y_{k+1}-q_k)+v_k^\tp(y_{k+1}-q_k)+\frac{L}{2}\norm{y_{k+1}-q_k}^2\right]\notag\\
&\le \frac{\tau_2}{2L}\mathbb{E}\norm{v_k-\nabla f(q_k)}^2+\frac{L}{2}\left(1+\frac{1}{\tau_2}\right)\mathbb{E}\norm{y_{k+1}-q_k}^2+\mathbb{E} v_k^\tp(y_{k+1}-q_k)\notag\\
&\le \tau_2(\mathbb{E}f(\tilde{x})-\mathbb{E}f(q_k)-\mathbb{E}v_k^\tp (\tilde{x}-q_k))+\frac{L}{2}\left(1+\frac{1}{\tau_2}\right)\mathbb{E}\norm{y_{k+1}-q_k}^2+\mathbb{E}v_k^\tp(y_{k+1}-q_k),\tag{S13}
\end{align}
\end{subequations}
where the second-last inequality follows from the identity $a^Tb \le \tfrac12 \|a\|^2 + \tfrac12 \|b\|^2$, and 
the final step follows from the so-called variance upper bound in the literature (Lemma~3.4 of \cite{katyusha2016}).

To prove Lemma \ref{lem:SupplyKatyusha}, we need to show that 
\begin{align}\tag{S14}
\label{eq:sup1}
(\mathbb{E}F(y_{k+1})-F(x_\star))-(1-\tau_1-\tau_2) (\mathbb{E}F(y_k)-F(x_\star))-\tau_2 (\mathbb{E}F(\tilde{x})-F(x_\star))\le -\mathbb{E} S_1(\xi_k, w_k).
\end{align}
For brevity, define $\tilde\tau \defeq 1-\tau_1-\tau_2$. The left side of \eqref{eq:sup1} can be rewritten as
\begin{align}
  \nonumber
  &(\mathbb{E}F(y_{k+1})-F(x_\star))-(1-\tau_1-\tau_2) (\mathbb{E}F(y_k)-F(x_\star))-\tau_2 (\mathbb{E}F(\tilde{x})-F(x_\star))\\
  \nonumber
  &= \mathbb{E} f(y_{k+1})+\mathbb{E} \psi(y_{k+1})-\tilde\tau \mathbb{E} f(y_k)-\tilde\tau \mathbb{E} \psi(y_k)-\tau_1 f(x_\star)-\tau_1\psi(x_\star)-\tau_2 \mathbb{E} f(\tilde{x})-\tau_2 \mathbb{E} \psi(\tilde{x})\\
  \label{eq:s12a}
 &= \big(\mathbb{E} f(y_{k+1})-\tilde\tau \mathbb{E} f(y_k)-\tau_1 f(x_\star)-\tau_2 \mathbb{E} f(\tilde{x})\big)+\big(\mathbb{E} \psi(y_{k+1})-\tilde\tau \mathbb{E} \psi(y_k)-\tau_1\psi(x_\star)-\tau_2 \mathbb{E} \psi(\tilde{x})\big).\tag{S15}
\end{align}
We have decoupled the left side of \eqref{eq:sup1} into the sum of two
terms, the first involving only $f$, and the second involving only
$\psi$. We will use the properties of $f$ and $\psi$ to provide upper
bounds in the quadratic forms for the first and second terms,
respectively.

Bounding the first term in (S15), we obtain
\begin{align}
&\mathbb{E} f(y_{k+1})-\tilde\tau \mathbb{E} f(y_k)-\tau_1 f(x_\star)-\tau_2 \mathbb{E} f(\tilde{x}) \notag\\
&= \mathbb{E} \left[f(y_{k+1})-f(q_k)+\tau_2 (f(q_k)-f(\tilde{x}))+\tau_1(f(q_k)-f(x_\star))+\tilde\tau(f(q_k)-f(y_k))\right] \notag\\
&\le \frac{L}{2}\left(1+\frac{1}{\tau_2}\right)\mathbb{E}\norm{y_{k+1}-q_k}^2+\mathbb{E} v_k^\tp(y_{k+1}-q_k)+\tau_2\mathbb{E}v_k^\tp(q_k-\tilde{x})+\tau_1\mathbb{E}v_k^\tp (q_k-x_\star)+\tilde\tau \mathbb{E} v_k^\tp (q_k-y_k),\tag{S16}\label{eq:ineq_first2}
\end{align}
where the last step follows from the three bounds (S11), (S12), and
(S13).
Next, strong convexity of $\psi$ leads to an upper bound for the second term in \eqref{eq:s12a}:
\begin{align}
\label{eq:ineq_second1}
&\mathbb{E} \psi(y_{k+1})-\tilde\tau \mathbb{E} \psi(y_k)-\tau_1\psi(x_\star)-\tau_2 \mathbb{E} \psi(\tilde{x}) \notag \\
&=\mathbb{E}\left[\tilde\tau (\psi(y_{k+1})-\psi(y_k))+\tau_1(\psi(y_{k+1})-\psi(z_{k+1}))+\tau_1(\psi(z_{k+1})-\psi(x_\star))+\tau_2(\psi(y_{k+1})-\psi(\tilde{x}))\right] \notag \\
&\le \mathbb{E}\big[\tilde\tau\,h_k^\tp(y_{k+1}-y_k)+\tau_1\, h_k^\tp(y_{k+1}-z_{k+1})+\tau_1\left(g_k^\tp(z_{k+1}-x_\star)-\frac{\sigma}{2}\norm{z_{k+1}-x_\star}^2\right)+\tau_2\,h_k^\tp(y_{k+1}-\tilde{x})\big].\tag{S17}
\end{align}
Combining \eqref{eq:ineq_first2}--\eqref{eq:ineq_second1}, we see that the left side of \eqref{eq:sup1} is bounded above by the expected value of the following sum:
\begin{multline}\label{combosum}
\frac{L}{2}\left(1+\frac{1}{\tau_2}\right)\norm{y_{k+1}-q_k}^2+ v_k^\tp(y_{k+1}-q_k)+\tau_2\,v_k^\tp(q_k-\tilde{x})+\tau_1\,v_k^\tp (q_k-x_\star)+\tilde\tau\, v_k^\tp (q_k-y_k)\\
+ \tilde\tau\,h_k^\tp(y_{k+1}-y_k) + \tau_1\,h_k^\tp(y_{k+1}-z_{k+1})+\tau_1\left(g_k^\tp(z_{k+1}-x_\star)-\frac{\sigma}{2}\norm{z_{k+1}-x_\star}^2\right)+\tau_2\,h_k^\tp(y_{k+1}-\tilde{x}).\tag{S18}
\end{multline}

All terms in~\eqref{combosum} are actually quadratic forms,  due to the state-space model:
\begin{align*}
\bmat{z_{k+1}-x_\star \\ y_{k+1}-x_\star \\ \tilde{x} - x_\star}
&=\bmat{1 & 0 & 0\\ \tau_1 & \tilde\tau & \tau_2 \\ 0 & 0 & 1} \bmat{z_k-x_\star \\ y_k-x_\star \\ \tilde{x}-x_\star}+\bmat{-\alpha & -\alpha & 0 \\ -\zeta & 0 & -\zeta \\ 0  & 0 & 0}\bmat{v_k  \\ g_k \\ h_k }, \\
q_k-x_\star &= \bmat{\tau_1& \tilde\tau & \tau_2}\bmat{z_k-x_\star \\ y_k-x_\star \\ \tilde{x}-x_\star},
\end{align*}
where we recall the definition $\tilde\tau\defeq 1-\tau_1-\tau_2$.
For example, the term $v_k^\tp(y_{k+1}-q_k)$ is equivalent to the quadratic form:
\begin{align*}
\bmat{z_k-x_\star \\ y_k-x_\star \\ \tilde{x}-x_\star \\ v_k \\ g_k \\ h_k}^\tp   \left( \bmat{0 & 0 & 0 & 0 & 0 & 0 \\ 0 & 0 & 0 & 0 & 0 & 0 \\0 & 0 & 0 & 0 & 0 & 0 \\ 0 & 0 & 0 & -\zeta & 0 & -\frac{\zeta}{2} \\ 0 & 0 & 0 & 0 & 0 & 0 \\ 0 & 0 & 0 & -\frac{\zeta}{2} & 0 & 0 } \otimes I_p  \right)   \bmat{z_k-x_\star \\ y_k-x_\star \\ \tilde{x}-x_\star \\ v_k \\ g_k \\ h_k}.
\end{align*}
Summing all the these quadratic forms directly yields the desired supply rate.

\section{Guidelines for Constructing and Choosing Supply Rates}

In most cases, supply rates may be constructed by manipulating well-known quadratic inequalities. One can see this in the proof of Lemma \ref{lem:supply_SVRGI} and Lemma \ref{lem:supply_SVRGII}.
For momentum methods, the supply rate construction is more involved. One typically needs to regroup terms carefully after adding and subtracting $f(q_k)$, where $q_k$ is the input to the stochastic gradient. See \eqref{eq:ineq_first2} for such an example.  We note that it is possible for different supply rate functions to yield the same iteration complexity bound. It is also possible to construct other supply rate functions that yield a constant-factor improvement for the convergence guarantees of Katyusha. In the present work, we only provide one supply rate for the analysis of Katyusha.

The selections of supply rate functions for a particular algorithm can be guided by the numerical solutions of the proposed LMIs. For example, one could include several candidate supply rates with associated multipliers $\lambda_j$ in the LMI to identify which supply rate functions are needed to obtain the desired rate bound.

\end{document}